\documentclass[12pt,preprint,number,sort&compress]{elsarticle}
\usepackage{amssymb, amscd, amsmath, float, graphicx, longtable,
  color, enumerate, bm, stmaryrd,amsthm,todonotes} 
\usepackage{fouriernc}
\usepackage[pdftex]{hyperref}

\DeclareMathOperator{\depth}{depth}

\DeclareMathOperator{\diag}{diag}
\DeclareMathOperator{\End}{End}

\DeclareMathOperator{\id}{id}

\DeclareMathOperator{\Mat}{Mat}

\DeclareMathOperator{\tr}{tr}

\renewcommand{\phi}{\varphi}

\renewcommand{\bar}{\overline}
\renewcommand{\tilde}{\widetilde}

\renewcommand{\leq}{\leqslant}

\renewcommand{\geq}{\geqslant}

\renewcommand{\to}{\longrightarrow}

\newcommand{\m}{{\mathfrak{m}}}

\newcommand{\PP}{{\mathbb P}}

\newcommand{\calc}{{\mathcal C}}

\newcommand{\cale}{{\mathcal E}}

\newtheorem{theorem}{Theorem}

\newtheorem{prop}[theorem]{Proposition}

\newtheorem{lemma}[theorem]{Lemma}

\newtheorem*{theorem*}{Theorem}
\newtheorem*{prop*}{Proposition}
\newtheorem*{maintheorem}{Main Theorem}

\theoremstyle{definition}

\newtheorem{definition}[theorem]{Definition}

\newtheorem{remark}[theorem]{Remark}
\newtheorem{notation}[theorem]{Notation}

\newtheorem{example}[theorem]{Example}

\numberwithin{equation}{theorem}

\begin{document}
\begin{frontmatter}
  \long\def\symbolfootnote[#1]#2{\begingroup%
    \def\thefootnote{\fnsymbol{footnote}}\footnote[#1]{#2}\endgroup}
  \title{Wild hypersurfaces\symbolfootnote[0]{To Roger
      Wiegand, our advisor, on the occasion of his $N^\text{th}$
      ascent of the Diamond.}}

\author[su]{Andrew Crabbe}
\ead{andrew.m.crabbe@gmail.com}

\author[su]{Graham J. Leuschke\corref{cor1}\fnref{fn1}}
\ead{gjleusch@math.syr.edu}
\ead[url]{http://www.leuschke.org/}

\cortext[cor1]{Corresponding Author}
\fntext[fn1]{GJL was partly supported by NSF grant DMS-0902119.}

\address[su]{Department of Mathematics, Syracuse University,
Syracuse NY 13244, USA}



\date{\today}


\begin{abstract}
Complete hypersurfaces of dimension at least $2$ and multiplicity at
least $4$ have wild Cohen-Macaulay type.   
\end{abstract}
\end{frontmatter}

\section*{Introduction}\label{sect:intro}

Let $R$ be a (commutative, Noetherian) local ring.  A finitely
generated $R$-module $M$ is called \emph{maximal Cohen--Macaulay}
(MCM) provided $\depth M =  \dim R$.  In particular, $R$ is a
Cohen--Macaulay (CM) ring if it is MCM as a module over itself.

This paper is about \emph{CM representation types,} specifically tame
and wild CM types.  See \S\ref{sect:tamewild} for the definitions of
these properties.  In this Introduction, we motivate our main result
by recalling the classification of complete equicharacteristic
hypersurface rings of \emph{finite CM type.}

\begin{theorem*}
  [{\cite{Knorrer, BGS}}] Let $k$ be an algebraically closed field of
  characteristic not equal to $2$, $3$, or $5$.  Let $d \geq 1$, let
  $f \in k[\![x_0, \dots, x_d]\!]$ be a non-zero non-unit power
  series, and let $R = k[\![x_0, \dots, x_d]\!]/(f)$ be the
  corresponding hypersurface ring.  Then there are only finitely many
  isomorphism classes of indecomposable MCM $R$-modules if, and only
  if, we have an isomorphism $R \cong k[\![x_0, \dots,
  x_d]\!]/(g(x_0,x_1) + x_2^2 + \cdots + x_d^2)$, where $g(x_0,x_1)$
  is one of the following polynomials, indexed by the ADE
  Coxeter--Dynkin diagrams:
  \begin{equation*}\label{eq:ADE}
    \begin{split}
      (A_n):& \qquad x_0^2 + x_1^{n+1}\text{, \quad some }n \geq 1\,;\\
      (D_n):& \qquad x_0^2x_1 +x_1^{n-1}\text{, \quad some }n \geq 4\,;\\
      (E_6):& \qquad x_0^3 + x_1^4\,;\\
      (E_7):& \qquad x_0^3 + x_0 x_1^3\,;\\
      (E_8):& \qquad x_0^3 + x_1^5\,.
    \end{split}
  \end{equation*}
\end{theorem*}

A \emph{key step} in the proof of this theorem
is~\cite[Prop. 3.1]{BGS}, which says that if $d \geq 2$ and the
multiplicity $\operatorname{e}(R)$ is at least $3$ (equivalently $f
\in (x_0, \dots, x_d)^3$) then $R$ has a family of indecomposable MCM
modules parametrized by the points of a cubic hypersurface in
$\PP_k^{d}$.

One would like a classification theorem like the one above for, say,
hypersurfaces of tame CM type.  (Again, see \S\ref{sect:tamewild} for
definitions.)  Drozd and Greuel have shown~\cite{Drozd-Greuel:1993}
that the one-dimensional hypersurfaces defined in $k[\![x_0,x_1]\!]$
by 
\[
(T_{pq})\qquad x_0^p + x_1^q + \lambda x_0^2 x_1^2\,,
\]
with $p, q \geq 2$, $\lambda \in k \setminus \{0,1\}$, and $k$ an
algebraically closed field of characteristic not equal to $2$, have
tame CM type.  (With the exception of the cases $(p,q) = (4,4)$ and
$(3,6)$, one may assume $\lambda =1$.)  In fact, they show that a
curve singularity of infinite CM type has tame CM type if and only if
it birationally dominates one of these hypersurfaces.  More recently,
Drozd, Greuel, and Kashuba~\cite{Drozd-Greuel-Kashuba:2003} have shown
that the two-dimensional analogues
\[
(T_{pqr})\qquad x_0^p + x_1^q + x_2^r + x_0x_1x_2
\]
with $\frac1p + \frac 1q + \frac 1r \leq 1$ have tame CM type.  Since
these hypersurface rings have multiplicity $3$ in general, the desired
\emph{key step} in a classification of hypersurface rings of tame CM
type would have to be of the form ``If $d \geq 2$ and
$\operatorname{e}(R) \geq 4$, then $R$ has wild CM type.''  This
result is indeed true for $d=2$, as proved by
Bondarenko~\cite{Bondarenko:2007}.  

In working through Bondarenko's proof, we found a way to simplify the
argument somewhat; this simplification allows us to prove the desired
\emph{key step} for all $d \geq 2$.  Thus we prove
(Theorem~\ref{thm:main}) 

\begin{maintheorem}
  Let $S = k[\![x_0,\ldots,x_d]\!]$ with $d \geq 2$ and $f$ a non-zero
  power series of order at least $4$.  Then $R = S/(f)$ has wild
  Cohen-Macaulay type.
\end{maintheorem}

By the original \emph{key step} of~\cite{BGS}, the case $d \geq 3$ is
already known to admit at least a $\PP^2$ of indecomposable MCM
modules, so is already perhaps known by experts to have wild type.
Not being aware of an explicit statement to that effect, we think that
a unified statement is desirable.

In \S\ref{sect:tamewild} we give a brief survey of tame and wild
representation types for the commutative-algebraist reader, including
Drozd's proof of the essential fact that $k[a_1, \dots, a_n]$ is
finite-length wild for $n\geq2$, and in \S\ref{sect:main} we prove the
Main Theorem.

We are grateful to the anonymous referee, whose careful reading
improved the paper.

\section{Tameness and Wildness}\label{sect:tamewild}

There are several minor variations on the notions of tame and wild
representation type, but the intent is always the same: tame
representation type allows the possibility of a classification theorem
in the style of Jordan canonical form, while for wild type any
classification theorem at all is utterly out of reach.  The
definitions we will use are essentially those of
Drozd~\cite{Drozd:1977}; they seem to have appeared implicitly first
in~\cite{Donovan-Freislich:1973}.  They make precise the intent
mentioned above by invoking the classical unsolved problem of
canonical forms for $n$-tuples of matrices up to simultaneous
similarity~\cite{Gelfand-Ponomarev:1969} (see
Example~\ref{eg:twovars} below).

\begin{definition}
  Let $k$ be an infinite field, $R$ a local $k$-algebra, and let
  $\calc$ be a full subcategory of the finitely generated $R$-modules.
  \begin{enumerate}[\quad(i)]
  \item We say that $\calc$ is \emph{tame,} or \emph{of tame
      representation type,} if there is one discrete parameter $r$
    (such as $k$-dimension or $R$-rank) parametrizing the
    modules in $\calc$, such that, for each $r$, the indecomposables
    in $\calc$ form finitely many one-parameter families and finitely
    many exceptions.  Here a \emph{one-parameter family} is a set of
    $R$-modules $\{E/(t-\lambda)E\}_{\lambda \in k}$, where $E$ is a
    fixed $k[t]$-$R$-bimodule which is finitely generated and free over
    $k[t]$.

  \item We say that $\calc$ is \emph{wild,} or \emph{of wild
      representation type,} if for every finite-dimensional
    $k$-algebra $\Lambda$ (not necessarily commutative!), there exists
    a \emph{representation embedding} $\cale \colon
    \operatorname{mod}\Lambda \to \calc$, that is, $\cale$ is an exact
    functor preserving non-isomorphism and indecomposability.
  \end{enumerate}

We are mostly interested in two particular candidates for $\calc$.
When $\calc$ consists of the full subcategory of $R$-modules of finite
length, then we say $R$ is \emph{finite-length tame} or
\emph{finite-length wild}.  At the other extreme, when $\calc$ is the
full subcategory $\operatorname{MCM}(R)$ of maximal Cohen--Macaulay
$R$-modules, we say $R$ has \emph{tame} or \emph{wild CM type.}
\end{definition}

The following Dichotomy Theorem justifies the slight unwieldiness of
the definitions.  (See also~\cite{Klingler-Levy:2006} for a more
general statement.)

\begin{theorem}
  [Drozd~\cite{Drozd:1977,Drozd:1979}, Crawley-Boevey~\cite{CrawleyBoevey:1988}]
  A finite-dimensional algebra over an algebraically closed field is
  either finite-length tame or finite-length wild, and not both.
\end{theorem}

In this paper we will be most concerned with wildness.  It follows
immediately from the definition that, to establish that a given module
category $\calc$ is wild, it suffices to find a single particular
example of a wild $\calc_0$ and a representation embedding $\calc_0
\to \calc$.  To illustrate this idea, as well as for our own use in the
proof of the Main Theorem, we give here a couple of examples.

\begin{example}
  [{\cite{Gelfand-Ponomarev:1969}}]  
  \label{eg:twovars}
  The non-commutative polynomial ring $k\langle a,b\rangle$ over an
  infinite field $k$ is finite-length wild.  To see this, let $\Lambda
  = k\langle x_1, \dots, x_m\rangle/I$ be an arbitrary
  finite-dimensional $k$-algebra and let $V$ be a $\Lambda$-module of
  finite $k$-dimension $n$.  Represent the actions of the variables
  $x_1, \dots, x_m$ on $V$ by linear operators $X_1, \dots, X_m
  \in \End_k(V)$.  For $m$ distinct scalars $c_1, \dots, c_m \in k$,
  define a $k\langle a,b\rangle$-module $M = M_V$ as follows: the
  underlying vector space of $M$ is $V^{(m)}$, and we let $a$ and $b$
  act on $M$ via the linear operators
  \[
  A=\left[\begin{matrix}c_1\id_V \\ & c_2\id_V \\ && \ddots \\ &&&
      c_m\id_V\end{matrix}\right]
  \qquad\text{and}\qquad
  B=\left[\begin{matrix}X_1 \\ \id_V & X_2 \\ &\ddots & \ddots \\ && \id_V&
      X_m\end{matrix}\right]\,,  
  \]
  respectively.

  A homomorphism of $k\langle a,b\rangle$-modules from $M_{V}$ to
  $M_{V'}$ is defined by a vector space homomorphism $S \colon V^{(m)}
  \to {V'}^{(m)}$ satisfying $SA = A'S$ and $SB = B'S$, where $A$ and
  $B$, resp.\ $A'$ and $B'$, are the matrices defining the $k \langle
  a, b \rangle$ structures on $M_V$, resp.\ $M_{V'}$.  Two modules
  $M_{V}$ and $M_{V'}$ are isomorphic via $S$ if and only if $S$ is
  invertible over $k$.  Similarly, a module $M_V$ is decomposable if
  and only if there is a non-trivial idempotent endomorphism $S \colon
  M_V \to M_V$.
  
  Assume that $\dim_k V = \dim_k V'$ and let $S \colon V^{(m)} \to
  {V'}^{(m)}$ be a vector space homomorphism such that $S A = A' S$
  and $S B = B' S$.  Then we can write $S = (\sigma_{ij})_{1 \leq i,j
    \leq m}$, with each $\sigma_{ij} \colon V \to V'$; we will show
  that $S = \diag(\sigma_{11}, \dots, \sigma_{11})$ and
  that $\sigma_{11} X_i = X_i' \sigma_{11}$ for each $i=1, \dots, m$.
  Thus $S$ is an isomorphism if and only if $\sigma_{11} \colon V \to
  V'$ is an isomorphism of $\Lambda$-modules, and $S$ is idempotent if
  and only if $\sigma_{11}$ is.

  The equation $SA = A' S$ implies $\sigma_{ij}c_j = c_i \sigma_{ij}$
  for every $i,j$.  Since the scalars $c_i$ are pairwise distinct,
  this implies that $\sigma_{ij}=0$ for all $i\neq j$, so that $S$ is
  a block-diagonal matrix.  Now the equation $S B = B' S$ becomes
  \footnotesize
  \[
    \begin{bmatrix}
    \sigma_{11}X_1 \\ 
    \sigma_{22} &\ddots \\ 
    &  \ddots &\sigma_{m-1,m-1}X_{m-1} \\ 
    &  & \sigma_{mm}& \sigma_{mm} X_m
    \end{bmatrix}
  =
    \begin{bmatrix}
      X_1'\sigma_{11} \\ 
      \sigma_{11} & \ddots\\ 
      &  \ddots  &X_{m-1}'\sigma_{m-1,m-1} \\ 
      & & \sigma_{m-1,m-1}&X_m' \sigma_{mm}
      \end{bmatrix}\,,
  \]
  \normalsize which implies that $\sigma_{ii}=\sigma_{11}$ for each
  $i=1, \dots, m$.  Denote the common value by $\sigma$; then the
  diagonal entries show that $\sigma X_i = X_i' \sigma$ for each $i=1,
  \dots, m$.
\end{example}

\begin{example}
  [{\cite{Drozd:1972}}]
  \label{eg:Drozdring}
  Let $k$ be an infinite field, and set $R = k[a,b]/(a^2,ab^2,b^3)$.
  Then $R$ is finite-length wild.  Consequently, the commutative
  polynomial ring $k[a_1, \dots, a_n]$ and the commutative power
  series ring $k[\![a_1, \dots, a_n]\!]$ are both finite-length wild
  as soon as $n \geq 2$.

  The last sentence follows from the one before, since any $R$-module
  of finite length is also a module of finite length over $k[a,b]$ and
  $k[\![a,b]\!]$, whence also over $k[a_1, \dots, a_n]$ and $k[\![a_1,
  \dots, a_n]\!]$.  Thus by Example~\ref{eg:twovars} above, it
  suffices to construct a representation embedding of the
  finite-length modules over $k\langle x,y\rangle$ into
  $\operatorname{mod} R$.

  Let $V$ be a $k\langle x,y\rangle$-module of $k$-dimension $n$, with
  linear operators $X$ and $Y$ representing the $k\langle
  x,y\rangle$-module structure.  We define $(32 n \times 32
  n)$ matrices $A$ and $B$ yielding an $R$-module structure on $M =
  M_V = V^{(32)}$.  To wit, let $c_1, \dots, c_5 \in k$ be distinct
  scalars and
  \[
  A = 
  \left[\begin{matrix}0 & 0 & \id_{V^{(15)}} \\ 
      0 & 0 & 0 \\
      0 & 0 & 0
      \end{matrix}\right]
    \qquad\text{ and }\qquad
  B = 
  \begin{bmatrix}B_1 & 0 & B_2 \\
    0 & 0 & B_3 \\
    0 & 0 & B_1
    \end{bmatrix}\,,
  \]
  where
  \[
    B_1 = 
    \begin{bmatrix}
      0 & 0 & \id_{V^{(5)}} \\ 
      0 & 0 & 0 \\
      0 & 0 & 0
    \end{bmatrix}\,,
    \qquad  
    B_2 =
    \begin{bmatrix}
      0 & 0 & 0 \\
      \id_{V^{(5)}} & 0 & 0\\
      0 & C & 0 
    \end{bmatrix}\,,
    \quad\text{ and }\quad
    B_3 = 
  \begin{bmatrix}
    0 & D & 0 \\
  \end{bmatrix}\,,
  \]
  and finally
  \[
  C = 
  \begin{bmatrix}
    c_1\id_V \\ & c_2\id_V \\ && \ddots \\ &&&
      c_5\id_V
  \end{bmatrix}
  \quad\text{ and }\quad
  D = 
  \begin{bmatrix}
    \id_V & 0 & \id_V & \id_V & \id_V\\
    0 & \id_V & \id_V & X & Y
  \end{bmatrix}\,.
  \]
  Observe that, while all the blocks in $B_1$, $B_2$, and $B_3$ are
  $(5n \times 5n)$, the blocks in $A$ and $B$ are not of uniform
  size; their four corner blocks are $(15n \times 15n)$, while the
  center block is $(2n \times 2n)$.

  One verifies easily that $AB=BA$ and $A^2 = AB^2 = B^3=0$, so $A$ and
  $B$ do indeed define an $R$-module structure on $M_V$.  

  Let $V'$ be a second $n$-dimensional $k\langle x,y\rangle$-module,
  with linear operators $X'$ and $Y'$ defining the $k\langle
  x,y\rangle$-module structure, and define $M' = M_{V'}$ as above,
  with linear operators $A'$ and $B'$ giving $M'$ the structure of an
  $R$-module.  Let $S \colon V^{(32n)} \to {V'}^{(32n)}$ be a vector
  space homomorphism such that $SA=A'S$ and $SB=B'S$.  We will show
  that in this case $S$ is a block-upper-triangular matrix (with
  blocks of size $n$) having constant diagonal block $\sigma \colon V
  \to V'$ which satisfies $\sigma X = X' \sigma$ and $\sigma Y = Y'
  \sigma'$.  Thus $S$ is an isomorphism of $R$-modules if and only if
  $\sigma$ is an isomorphism of $k\langle x,y\rangle$-modules, and $S$
  is a split surjection if and only if $\sigma$ is so.  It follows
  that the functor $V \rightsquigarrow M_V$ is a representation
  embedding, and $R$ is finite-length wild.

  Note that $A$ is independent of the module $V$, so $A=A'$ and
  $SA=AS$.  Write $S$ in block format, with blocks of the same sizes
  as $A$, 
  \[
  S = \begin{bmatrix}
    S_{11} & S_{12} & S_{13} \\
    S_{21} & S_{22} & S_{23} \\
    S_{31} & S_{32} & S_{33} 
  \end{bmatrix}\,,
  \]
  this means
  \[
  \begin{bmatrix}
    0 & 0 & S_{11} \\
    0 & 0 & S_{21} \\
    0 & 0 & S_{31} \\
  \end{bmatrix}
  =
  \begin{bmatrix} 
    S_{31} & S_{32} & S_{33} \\
    0 & 0 & 0 \\
    0 & 0 & 0 \\
  \end{bmatrix}
  \]
  so that 
  \[
  S = \begin{bmatrix}
    S_{11} & S_{12} & S_{13} \\
    0 & S_{22} & S_{23} \\
    0 & 0 & S_{11} 
  \end{bmatrix}\,.
  \]
  Using now the equation $SB=B'S$, we get 
  \[
  \begin{bmatrix}
    S_{11}B_1 & 0 & S_{11}B_2 + S_{12}B_3 + S_{13}B_1 \\
    0 & 0  & S_{22}B_3 + S_{23}B_1 \\
    0 & 0  & S_{11}B_1
  \end{bmatrix}
  =
  \begin{bmatrix}
    B_1'S_{11} & B_1'S_{12} & B_1'S_{13} + B_2'S_{11} \\
    0 & 0  & B_3'S_{11} \\
    0 & 0 & B_1'S_{11}
  \end{bmatrix}\,.
  \]
  In particular, $S_{11}B_1 = B_1' S_{11}$.  Write the $(15n\times
  15n)$ matrix $S_{11}$ in $(5n \times 5n)$-block format as
  \[
  S_{11} = \begin{bmatrix}
    T_{11} & T_{12} & T_{13} \\
    T_{21} & T_{22} & T_{23} \\
    T_{31} & T_{32} & T_{33} 
  \end{bmatrix}\,.
  \]
  Then the definition of $B_1$ and $B_1'$ gives 
  \[
  S_{11} = \begin{bmatrix}
    T_{11} & T_{12} & T_{13} \\
    0  & T_{22} & T_{23} \\
    0  & 0 & T_{11} 
  \end{bmatrix}
  \]
  as above.  Now $S_{12}$ is $(15n \times 2n)$, so we write it in $(5n
  \times 2n)$ blocks as 
  $S_{12} = \begin{bmatrix}U_1 & U_2 & U_3 \end{bmatrix}^{\tr}$
  and use $B_1' S_{12} =0$ to get 
  $S_{12} = \begin{bmatrix}U_1 & U_2 & 0 \end{bmatrix}^{\tr}$.  We
  also have $S_{22} B_3 + S_{23} B_1 = B_3'S_{11}$; if we write
  $S_{23} = \left[\begin{matrix} V_1 & V_2 & V_3 \end{matrix}\right]$,
  then this equation reads
  \[
  \begin{bmatrix} 0 & S_{22}D & 0 \end{bmatrix}
  +
  \begin{bmatrix} 0 & 0 & V_1 \end{bmatrix}
  =
  \begin{bmatrix} 0 & D'T_{22} & D'T_{23} \end{bmatrix}\,.
  \]
  It follows that $S_{22}D = D'T_{22}$ and $S_{23} =
  \left[\begin{matrix} D'T_{23} & V_2 &
      V_3 \end{matrix}\right]$.
  
  Finally write
  \[
  S_{13} = 
  \begin{bmatrix}
    W_{11} & W_{12} & W_{13} \\
    W_{21} & W_{22} & W_{23} \\
    W_{31} & W_{32} & W_{33} 
  \end{bmatrix}
  \]
  and consider the equation
  \[
  S_{11}B_2 + S_{12}B_3 + S_{13}B_1 = B_1'S_{13} + B_2'S_{11}
  \,.
  \]
  It becomes
  \[
  \begin{bmatrix}
    T_{12} & T_{13}C + U_1 D & W_{11} \\
    T_{22} & T_{23}C + U_2 D & W_{21} \\
    0 & T_{11}C & W_{31} 
  \end{bmatrix}
  =
  \begin{bmatrix}
    W_{31} & W_{32} & W_{33} \\
    T_{11} & T_{12} & T_{13} \\
    0 & C' T_{22} & C' T_{23}
  \end{bmatrix}\,.
  \]
  We read off $T_{22} = T_{11}$ and $T_{11} C = C' T_{11}$.  Since $C=C'$
  is a diagonal matrix with distinct blocks $c_1 \id_V, \dots, c_5
  \id_V$, this forces $T_{11}$ to be block-diagonal, 
  \[
  T_{11} = 
  \begin{bmatrix}
    Z_1 \\ & \ddots \\ & & Z_5
  \end{bmatrix}\,,
  \]
  with each $Z_i$ an $(n \times n)$ matrix.

  We also have $S_{22} D = D' T_{11}$.  Write $S_{22} =
  \left[\begin{smallmatrix} E & F \\ G & H \end{smallmatrix}\right]$
  so that this reads
  \small
  \[
  \begin{bmatrix}
    E & F \\ G & H
  \end{bmatrix}
  \begin{bmatrix}
    \id_V & 0 & \id_V & \id_V & \id_V \\
    0 & \id_V & \id_V & X & Y
  \end{bmatrix}
  =
  \begin{bmatrix}
    \id_V & 0 & \id_V & \id_V & \id_V \\
    0 & \id_V & \id_V & X' & Y'
  \end{bmatrix}  
  \begin{bmatrix}
    Z_1 \\ & \ddots \\ && Z_5
  \end{bmatrix}\,.
  \]
  \normalsize
  Carrying out the multiplication, we conclude that $F=G=0$, so that
  $E = Z_1 = Z_3 = Z_4 = Z_5 $ and $H = Z_2 = Z_3$.  Set $\sigma = E
  =H$.  Then $H X = X' Z_4$ and $HY=Y' Z_5$ imply $\sigma X = X' \sigma$
  and $\sigma Y = Y' \sigma$, so that $\sigma$ is a homomorphism of
  $k\langle x,y\rangle$-modules $V \to V'$.  Since $T_{11}$ and
  $S_{22}$ are both block-diagonal with diagonal block $\sigma$, we
  conclude that $S$ is block-upper-triangular with constant diagonal
  block $\sigma$, as claimed.  
\end{example}

We restate one part of this example separately for later use.

\begin{prop}
  \label{prop:polyring}
  Let $Q = k[a_1, \dots, a_n]$ or $k[\![a_1, \dots, a_n]\!]$, with $n
  \geq 2$.  If there is a representation embedding of the
  finite-length $Q$-modules into a module category $\calc$, then
  $\calc$ is wild. \qed
\end{prop}

\section{Proof of the Main Theorem}\label{sect:main}

We use without fanfare the theory of matrix factorizations, namely the
equivalence between  matrix factorizations of a power series
$f$ and MCM modules over the hypersurface ring defined by $f$
(\cite{Eisenbud:1980}, see~\cite{Yoshino:book} for a complete
discussion).  The two facts we will use explicitly are contained in
the following Remark and Example.

\begin{remark}\label{sit:KnorFunc}
  Let $S$ be a regular local ring and $f \in S$ a non-zero non-unit.
  Set $T = S[\![u,v]\!]$.  Then the functor from matrix factorizations
  of $f$ over $S$ to matrix factorizations of $f+uv$ over $T$, defined
  by 
  \[
  (\phi, \psi) \mapsto \left(\left[
      \begin{matrix}\phi & -v I \\ u I
        & \psi \end{matrix} 
    \right], \left[
      \begin{matrix} \psi & v I \\ -u I & \phi\end{matrix}
    \right]\right)\,,
  \] 
  induces an equivalence of stable
  categories~\cite[Theorem 12.10]{Yoshino:book}.  In particular it gives a
  bijection on isomorphism classes of MCM modules over $S/(f)$ and
  $T/(f+uv)$.
\end{remark}
  
\begin{example}
  Let $k$ be a field and set $S_n =
  k[\![x_1,\ldots,x_n,y_1,\ldots,y_n]\!]$ and $f_n = x_1 y_1 + \cdots
  + x_n y_n$ for $n \geq 1$.  The ring $R_n = S_n/(f_n)$ is an ($A_1$)
  hypersurface singularity, so has finite Cohen-Macaulay type; in
  fact, there is only one non-free indecomposable MCM $R_n$-module, or
  equivalently, one nontrivial indecomposable matrix factorization of
  $f_n$. By the remark above, the non-trivial indecomposable matrix
  factorizations of $f_n$ are in bijection with those of $f_{n+1}$.
  For $n = 1$, the element $f_1 = x_1 y_1$ has only one nontrivial
  indecomposable matrix factorization up to equivalence, namely that
  represented by $(\phi_1, \psi_1) = (x_1, y_1)$. Defining
  \[
  (\phi_i, \psi_i) = \left(\left[
      \begin{matrix}\phi_{i-1} & -y_i I \\ x_i I
        & \psi_{i-1} \end{matrix} 
    \right], \left[
      \begin{matrix} \psi_{i-1} & y_i I \\ -x_i I &
        \phi_{i-1} \end{matrix} \right]\right)\,,
  \] 
  we have that $(\phi_n, \psi_n)$ represents the sole nontrivial
  indecomposable matrix factorization of $f_n$ over $S_n$.
\end{example}

Next we see that, at the cost of introducing some parameters, every
power series of sufficiently high order can be written in the form of
an ($A_1$) singularity, with some control over the coefficients.

\begin{lemma}\label{lem:A1}
  Let $f \in k[\![x_1, \dots, x_n,z]\!]$ be a power series of order at
  least $4$, and let $a_1, \dots, a_n$ be parameters.  Then $f$ can be written
  in the form
  \begin{equation}\label{eq:A1}
  f = z^2h + (x_1-a_1z)g_1 + \cdots + (x_n-a_n z)g_n
  \end{equation}
  where $g_1, \dots, g_n,h$ are power series in $x_1, \dots, x_n,z$
  with coefficients involving the parameters $a_1, \dots, a_n$, each
  $g_i$ has order at least $3$ in $x_1, \dots, x_n,z$, and $h$ has
  order at least $2$ in $x_1, \dots, x_n,z$.
\end{lemma}

\begin{proof}
  Work over $k[\![x_1,\dots, x_n,z]\!]$, with the parameters $a_1,
  \dots, a_n$ considered as variable elements of $k$, and consider the
  ideals $\m = (x_1,\dots, x_n,z)$ and $I = (x_1-a_1z, \dots, x_n-a_n
  z)$.  We claim that $(z^2)+I\m = \m^2$.  The left-hand side is
  clearly contained in the right.  For the other inclusion, simply
  check each monomial of degree $2$: $z^2 \in (z^2)+I\m$ by
  definition, whence
  \[
  x_i z = (x_i-a_i z)z+a_i z^2 \quad \in (z^2)+I\m
  \]
  for each $i$, and 
  \[
  x_i x_j = (x_i-a_i z)x_j + a_ix_j z \quad \in (z^2)+I\m
  \]
  for each $i,j$.  Writing $\m^4 = \m^2\m^2 = ((z^2)+I\m)\m^2 =
  z^2\m^2 + I\m^3$ completes the proof.
\end{proof}

Given an expression for $f \in k[\![x_1, \dots, x_n,z]\!]$ as in
Lemma~\ref{lem:A1}, we obtain from Remark~\ref{sit:KnorFunc} a matrix
factorization $(\phi_n,\psi_n)$ of $f$, with
\[
(\phi_0, \psi_0) = \left(\left[z^2\right],\left[h\right]\right)\,, 
\qquad\qquad
(\phi_1, \psi_1) = 
  \left(\left[
      \begin{matrix}z^2 & -g_1 \\ x_1-a_1z & h\end{matrix}
      \right], \left[
      \begin{matrix}h & g_1 \\ -x_1+a_1z & z^2\end{matrix}
      \right]\right)
\]
and, in general,
\[
(\phi_n, \psi_n) = 
  \left(\left[
      \begin{matrix}\phi_{n-1} & -g_n \id_{2^{n-1}} \\ (x_n-a_n z) \id_{2^{n-1}}
        & \psi_{n-1}\end{matrix} 
      \right], \left[
      \begin{matrix} \psi_{n-1} & g_n \id_{2^{n-1}} \\ (-x_n+a_n z)
        \id_{2^{n-1}} & \phi_{n-1}\end{matrix} 
      \right]\right)\,.
\]

We now describe how to ``inflate'' these matrix factorizations given a
$k[a_1,\dots, a_n]$-module of finite length.

\begin{definition}\label{def:inflation}
  Let $A_1, \dots, A_r$ be pairwise commuting $m \times m$ matrices
  over the field $k$.  Let $f = f(a_1, \dots, a_r)$ be a power series in
  variables $x_1, \dots, x_n,z$ involving the parameters $a_1, \dots,
  a_r$, which we think of as variable elements of $k$.  Let $F=F(A_1,
  \dots, A_r)$ be the $m \times m$ matrix obtained by replacing in $f$
  each scalar $\alpha \in k$ by $\alpha \id_m$, each $x_i$ by $x_i
  \id_m$, $z$ by $z \id_m$, and each parameter $a_i$ by the
  corresponding matrix $A_i$.  We call this process \emph{inflating}
  $f$.

  If $(\phi,\psi) = (\phi(a_1, \dots, a_r), \psi(a_1, \dots, a_r))$ is
  a matrix factorization, again involving parameters $a_1, \dots,
  a_r$, of an element $f \in k[\![x_1, \dots, x_n,z]\!]$, let $(\Phi,
  \Psi) = (\Phi(A_1, \dots, A_r), \Psi(A_1, \dots, A_r))$ be the
  result of inflating each entry of $\phi$ and $\psi$.  
\end{definition}

Note that in the second half of the definition, $f$ does not involve
the parameters.  It's easy to check that, since the $A_i$ commute,
$(\Phi,\Psi)$ is again a matrix factorization of $f$.  

It follows from Lemma~\ref{lem:A1} that a power series $f \in
k[\![x_1, \dots, x_n,z]\!]$ of order at least $4$ has, for every
$n$-tuple of commuting $m\times m$ matrices $(A_1, \dots, A_n)$ over
$k$, a matrix factorization
\begin{equation}\label{eq:biggen}
(\Phi_n,\Psi_n) = (\Phi(A_1, \dots, A_n), \Psi(A_1, \dots, A_n))
\end{equation}
of size $m 2^n$.

\begin{notation}
  Let $E = [e_{ij}]$ be a matrix with entries in $k[\![x_1, \dots,
  x_n,z]\!]$.  We set $\bar E = [\bar{e_{ij}}]$, where $\bar{e_{ij}}$
  denotes the image of $e_{ij}$ modulo the square of the maximal ideal
  $(x_1, \dots, x_n,z)$.

  Also, given a monomial $w \in k[\![x_1, \dots, x_n,z]\!]$, let
  $E\{w\}$ denote the matrix $[e_{ij}\{w\}]$, where $e_{ij}\{w\}$ denotes
  the coefficient of $w$ in the power series expansion of $e_{ij}$. We
  call this the ``$w$-strand'' of the matrix $E$.
\medskip

  For the rest of the paper, we let $f$ be a power series of order at
  least $4$ as in Lemma~\ref{lem:A1}, let $A_1, \dots, A_n$ and $A_1',
  \dots, A_n'$ be $n$-tuples of commuting $m\times m$ matrices over
  $k$, and let $(\Phi_n, \Psi_n)= (\Phi(A_1, \dots, A_n), \Psi(A_1,
  \dots, A_n))$ and $(\Phi'_n, \Psi'_n) = (\Phi(A_1', \dots, A_n'),
  \Psi(A_1', \dots, A_n'))$ be inflated matrix factorizations of $f$
  as in~\eqref{eq:biggen}.
\end{notation}

\begin{lemma}\label{lem:fourbigblocks}
  Let $i \in \{0, \dots, n\}$ and let $C$, $D$ be two $(m
  2^i\times m2^i)$ matrices with entries in $k$.  If $C$ and
  $D$ satisfy
  \begin{equation*}\tag{$\dagger$}\label{eq:barinduct}
     C \bar{\Phi_i} = \bar{\Phi_i'}D 
     \qquad\text{ and }\qquad
     D\bar{\Psi_i} = \bar{\Psi_i'}C\,,
  \end{equation*}
  then
  \begin{enumerate}[\quad$(i)$]
  \item\label{item:induct1} $C$ and $D$ are $(m \times m)$-block lower
    triangular, i.e.\ of the form
      \[
      C =   
      \left[\begin{matrix} C_{11} \\ &C_{22}&&\text{\huge{0}} \\ && \\
          &&&\ddots \\ &\text{\Huge{*}}&&&C_{2^i,2^i} 
        \end{matrix}\right]\,,\qquad
      D = 
      \left[\begin{matrix} D_{11} \\ &D_{22}&&\text{\huge{0}} \\ && \\
          &&&\ddots \\ &\text{\Huge{*}}&&&D_{2^i,2^i} 
        \end{matrix}\right]\,.
      \]
    \item\label{item:induct2} For each $j = 1, \dots, 2^i$,
      $C_{j j}$ and $D_{j j}$ are in the set $\{C_{11},\ D_{11}\}$.
    \item\label{item:induct3} For each $j = 1, \ldots, i$,
      $C_{2^i,2^i}A_{j} = A'_{j}D_{2^i-2^{j-1},2^i-2^{j-1}}$. 
  \end{enumerate}
\end{lemma}

\begin{proof}
  For parts (\ref{item:induct1}) and (\ref{item:induct2}), we proceed
  by induction on $i$.  The base case $i=0$ is vacuous.  For the
  inductive step, since in \eqref{eq:A1} $g_i \in (x_1, \dots,
  x_n,z)^3$, we can express $\bar{\Phi_i}$, $\bar{\Psi_i}$ as
  \[
  \bar{\Phi_i}=\left[
      \begin{matrix}\bar{\Phi_{i-1}} & 0 \\
        (\bar{x_i}\id_m-A_n\bar{z})\id_{2^{i-1}} 
        & \bar{\Psi_{i-1}}\end{matrix} 
  \right],\qquad \bar{\Psi_i}=\left[
      \begin{matrix}\bar{\Psi_{i-1}} & 0 \\
        (-\bar{x_i}\id_m+A_n\bar{z})\id_{2^{i-1}} 
        & \bar{\Phi_{i-1}}\end{matrix} 
      \right]
  \]
  and $\bar{\Phi_i'}$, $\bar{\Psi_i'}$ similarly, matrices over
  $k[\![x_1, \dots, x_n,z]\!]/(x_1, \dots, x_n,z)^2$.  (We write 
  $(\bar{x_i}\id_m-A_n\bar{z})\id_{2^{i-1}}$ to represent a $(m 2^{i-1}
  \times m 2^{i-1})$-block matrix with diagonal blocks
  $\bar{x_i}\id_m-A_n\bar{z}$.) 
Also express $C$ and $D$ in terms of their $(m 2^{i-1}
  \times m 2^{i-1})$-blocks
  \[
  C=\left[
      \begin{matrix} \gamma_{11} & \gamma_{12} \\ \gamma_{21} &
        \gamma_{22}\end{matrix}  
  \right],\qquad D=\left[
      \begin{matrix} \delta_{11} & \delta_{12} \\ \delta_{21} &
        \delta_{22}\end{matrix}  
      \right]\,.
  \]
  From $C \bar{\Phi_{i}} = \bar{\Phi_i'}D$, we get the equations
  \begin{align}
    &\gamma_{11}\bar{\Phi_{i-1}} + \gamma_{12}(\bar{x_i}\id_m - A_i
      \bar {z} ) \id_{2^{i-1}} = \bar{\Phi'_{i-1}}
      \delta_{11}\label{eq:fourbig1}\\
    &\gamma_{22}\bar{\Psi_{i-1}} = (\bar{x_i}\id_m - A_i'
      \bar{z})\id_{2^{i-1}} \delta_{12} + \bar{\Psi_{i-1}'}
      \delta_{22} \label{eq:fourbig2} \\
    &\gamma_{21} \bar{\Phi_{i-1}} + \gamma_{22}(\bar{x_i}\id_m - A_i
      \bar{z})\id_{2^{i-1}} = (\bar{x_i}\id_m -
      A_i'\bar{z})\id_{2^{i-1}} \delta_{11} +
      \bar{\Psi'_{i-1}}\delta_{21} \label{eq:fourbig3} \\
     \intertext{and from $D \bar{\Psi_i} = \bar{\Psi_i'}C$:}
    &\delta_{21} \bar{\Psi_{i-1}} + \delta_{22}(-\bar{x_i}\id_m +
       A_i\bar{z})\id_{2^{i-1}} = (-\bar{x_i}\id_m +
       A_i'\bar{z})\id_{2^{i-1}}\gamma_{11} +
       \bar{\Phi_{i-1}'}\gamma_{21}\,. \label{eq:fourbig4}
  \end{align}
  Since $\bar{\Phi_{i-1}}$, $\bar{\Psi_{i-1}}$, $\bar{\Phi_{i-1}'}$,
  $\bar{\Psi_{i-1}'}$ do not contain instances of $\bar{x_i}$, we conclude
  \begin{align*}
    &\text{from~\eqref{eq:fourbig1}:\qquad} \gamma_{12}=0 \quad\text{
      and }\quad \gamma_{11} \bar{\Phi_{i-1}} =
      \bar{\Phi_{i-1}'}\delta_{11}\,;\\
    &\text{from~\eqref{eq:fourbig2}:\qquad} \delta_{12}=0 \quad\text{
      and }\quad \gamma_{22} \bar{\Psi_{i-1}} =
      \bar{\Psi_{i-1}'}\delta_{22}\,;\\
    &\text{from~\eqref{eq:fourbig3}:\qquad} \gamma_{22}=\delta_{11}\,;
    \text{ and}\\
    &\text{from~\eqref{eq:fourbig4}:\qquad} \delta_{22}=\gamma_{11}\,.
  \end{align*}
  Thus the pair $\gamma_{11}$, $\delta_{11}$
  satisfy~\eqref{eq:barinduct}, so by the induction hypothesis, they
  satisfy~(\ref{item:induct1}) and~(\ref{item:induct2}).  Since $C$ and
  $\gamma_{11}$ share the same $(1,1)$ $m\times m$-block (and ditto
  for $D$ and $\delta_{11}$), the inductive proof is complete.
  
  For part (\ref{item:induct3}), we consider the $(m \times m)$ block
  in position $(2^i,2^i-2^{j-1})$ on either side of the equation $C
  \bar{\Phi_i} = \bar{\Phi_i'}D$.  We get that
  \[
  C_{2^i, 2^i} (\bar{x_j} \id_m - A_j \bar{z}) = (\bar{x_j} \id_m -
  A-_j \bar{z}) D_{2^i-2^{j-1},2^i-2^{j-1}}\,. 
  \]
  Examining the $\bar{z}$-strand yields the desired equality.
\end{proof}

\begin{prop}
  \label{prop:manysmallblocks}
  Let $(S,T) \colon (\Phi_n,\Psi_n) \to
  (\Phi_n',\Psi_n')$ be a homomorphism of matrix factorizations.  Then
  $S\{1\}$ and $T\{1\}$ are $(m\times m)$-block lower triangular of
  the form
  \begin{equation*}\tag{$*$}\label{eq:basicform}
  S\{1\} =   
  \left[\begin{matrix} U \\ &U&&\text{\huge{0}} \\ &&U \\ &&&\ddots \\
      &\text{\Huge{*}}&&&U 
    \end{matrix}\right]\,,\qquad
  T\{1\} = 
  \left[\begin{matrix} U \\ &U&&\text{\huge{0}} \\ &&U \\ &&&\ddots \\
      &\text{\Huge{*}}&&&U 
    \end{matrix}\right]\,,
  \end{equation*}
  where $U A_i = A'_i U$ for $i = 1,\ldots,n$.
\end{prop}

\begin{proof}
  We first show that $S_{11}\{1\} = T_{11}\{1\}$, where $S_{ij}$ and
  $T_{ij}$ denote the $(m \times m)$ blocks of $S$ and $T$,
  respectively, in the $(i,j)^\text{th}$ position.  For this, we
  consider the $(m\times m)$ block in position $(1,1)$ on either side
  of the equation $S \Phi_n = \Phi_n' T$.  We get that
  \[
  S_{11} z^2 + \sum_{i=1}^n S_{1, 2^{i-1}+1} (x_i \id_m - A_i z) = z^2
  T_{11} + \sum_{i=1}^n T_{2^{i-1}+1,1} G_i'\,,
  \]
  where $G_i'$ are the matrices resulting from ``inflating'' the power
  series $g_i'$.  Since the $g_i'$ have order at least $3$, each entry
  of $G_i'$ also has order at least $3$, and so the quadratic strands
  give the following equations:
  \begin{align}
    \{z^2\}\colon& \qquad S_{11}\{1\} - \sum_{i=1}^n
        S_{1,2^{i-1}+1}\{z\}A_i = T_{11}\{1\} \label{eq:manysmall1}\\
    \{x_i^2\}\colon& \qquad S_{1,2^{i-1}+1} \{x_i\} = 0 \label{eq:manysmall2}\\
    \{x_i z\}\colon& \qquad S_{1,2^{i-1}+1} \{z\} - \sum_{j=1}^n
        S_{1,2^{j-1}+1} \{x_i\} A_j = 0 \label{eq:manysmall3}\\
    \{x_i x_j\}\colon& \qquad S_{1,2^{j-1}+1} \{x_i\} +
    S_{1,2^{i-1}+1}\{x_j\}=0 \label{eq:manysmall4}
  \end{align}
  Starting from equation~\eqref{eq:manysmall1}, we have
  \[
  \begin{split}
    S_{11}\{1\} 
    &= T_{11}\{1\} + \sum_{i=1}^n S_{1,2^{i-1}+1}\{z\}A_i \\
    &= T_{11}\{1\} + \sum_{i=1}^n \left( \sum_{j=1}^n
        S_{1,2^{j-1}+1} \{x_i\} A_j A_i\right) \\
    &= T_{11}\{1\}\,,
  \end{split}
  \]
  the last equality following from
  equations~\eqref{eq:manysmall2},~\eqref{eq:manysmall4}, and the
  commutativity of the $A_i$.  
  
  The proof is completed by Lemma~\ref{lem:fourbigblocks} above. 
\end{proof}

Let $M$ be a $k[a_1,\ldots,a_n]$-module of dimension $m$ over $k$.
After choosing a $k$-basis for $M$, the action of each $a_i$ on $M$
can be expressed as multiplication by an $(m \times m)$ matrix $A_i$
over $k$.  (Note that the $A_i$'s must be pairwise commutative.)  We
may thus identify $M$ with the linear representation $L\colon
k[a_1,\ldots,a_n] \to \Mat_m(k)$, where $a_i \mapsto A_i$ for
$i=1,\ldots,n$.

A homomorphism from a linear representation $L(A_1,\ldots,A_n)$ to
another $L(A'_1,\ldots,A'_n)$ is defined by a matrix $U$ such that $U
A_i = A'_i U$ for $i = 1,\ldots,n$.  Representations are thus
isomorphic if this matrix $U$ is invertible.  A representation
$L(A_1,\ldots,A_n)$ is decomposable if it has a non-trivial idempotent
endomorphism, that is, there exists a matrix $U$ such that $U A_i =
A_i U$ for $i = 1,\ldots,n$, $U^2 = U$ and $U \neq 0,\ \id$.

\begin{theorem}\label{thm:main}
  Let $k$ be an infinite field, let $S = k[\![x_1, \dots, x_n,z]\!]$,
  and let $f \in S$ be a non-zero element of order at least $4$.  Set
  $R = S/(f)$.  Then the functor $F$ from finite-length $k[a_1, \dots,
  a_n]$-modules to MCM $R$-modules, sending a given linear
  representation $L(A_1,\ldots,A_n)$ to the inflated matrix
  factorization $(\Phi(A_1,\ldots,A_n),\Psi(A_1,\ldots,A_n))$, is a
  representation embedding.
  
  In particular, if $n \geq 2$ then $R$ has wild Cohen-Macaulay type.
\end{theorem}

\begin{proof}
  The functor $F$ is defined as follows on homomorphisms of
  linear representations $U \colon L(A_1, \dots, A_n) \to L(A_1',
  \dots, A_n')$.  If the $A_i$ are $\ell \times \ell$ matrices and the
  $A_i'$ are $m \times m$, then $U$ is an $m \times \ell$ matrix over
  $k$, and is sent to the block-diagonal $(m2^n \times \ell 2^n)$
  matrix $\tilde U$ with $U$ down the diagonal.  Since $U$ satisfies
  the relations $U A_i = A_i' U$, and the blocks of $\Phi(A_1, \dots,
  A_n)$ and $\Psi(A_1,\dots, A_n)$ are power series in the matrices
  $A_i$ with coefficients in $S$, we get $\tilde U \Phi = \Phi' \tilde
  U$ and $\tilde U \Psi = \Psi'\tilde U$.
  Now it is clear that $F$ is an exact functor.

  Suppose there is an isomorphism between matrix factorizations
  \[
  (S,T)\colon (\Phi(A_1,\ldots,A_n),\ \Psi(A_1,\ldots,A_n)) \rightarrow
  (\Phi(A'_1,\ldots,A'_n),\ \Psi(A'_1,\ldots,A'_n))\,.
  \] 
  By Proposition~\ref{prop:manysmallblocks}, $S\{1\}$ and
  $T\{1\}$ are of the form in~\eqref{eq:basicform}, in which  $U$
  defines a homomorphism from $L(A_1,\ldots,A_n)$ to
  $L(A'_1,\ldots,A'_n)$.  Since $S$ is invertible, so is $U$.
  Thus the representations are isomorphic.

  Suppose the matrix factorization
  $(\Phi(A_1,\ldots,A_n),\ \Psi(A_1,\ldots,A_n))$ is decomposable, that
  is, it has an endomorphism $(S,T)$ such that $S^2 =
  S$, $T^2 = T$ and $(S,T) \neq (0,0),
  (\id,\id)$.  Again, by Proposition~\ref{prop:manysmallblocks},
  $S\{1\}$ and $T\{1\}$ are of the form
  in~\eqref{eq:basicform}, in which the matrix $U$ now defines an
  idempotent endomorphism of the representation $L(A_1,\ldots,A_n)$.
  Since $S$ and $T$ are idempotent matrices, if $U = 0$, then
  $S=T = 0$.  Similarly, if $U = \id$, then
  $S=T = \id$.  Thus the representation $L(A_1,\ldots,A_n)$
  must be decomposable.

  The final sentence follows from Proposition~\ref{prop:polyring}.
\end{proof}

\renewcommand{\bibname}{References}


\newcommand{\arxiv}[2][AC]{\mbox{\href{http://arxiv.org/abs/#2}{\textsf{arXiv:#2[math.#1]}}}}
\newcommand{\oldarxiv}[2][AC]{\mbox{\href{http://arxiv.org/abs/math/#2}{\textsf{arXiv:math/#2[math.#1]}}}}
\providecommand{\MR}[1]{\mbox{\href{http://www.ams.org/mathscinet-getitem?mr=#1}{#1}}}
 \renewcommand{\MR}[1]{{\href{http://www.ams.org/mathscinet-getitem?mr=#1}{#1}}}
\providecommand{\bysame}{\leavevmode\hbox to3em{\hrulefill}\thinspace}
\providecommand{\MR}{\relax\ifhmode\unskip\space\fi MR}
\providecommand{\MRhref}[2]{%
 \href{http://www.ams.org/mathscinet-getitem?mr=#1}{#2}
}

\def\cprime{$'$} \def\cprime{$'$} \def\cprime{$'$}
\providecommand{\bysame}{\leavevmode\hbox to3em{\hrulefill}\thinspace}
\providecommand{\MR}{\relax\ifhmode\unskip\space\fi MR }
\providecommand{\MRhref}[2]{%
  \href{http://www.ams.org/mathscinet-getitem?mr=#1}{#2}
}
\providecommand{\href}[2]{#2}

\end{document}